\documentclass[]{amsart}

\usepackage{cite}

\usepackage{amsmath}
\usepackage{amsfonts}
\usepackage{amssymb}
\usepackage{amsthm}

\usepackage{rotating}

\newtheorem{theorem}{Theorem}
\newtheorem{prop}{Proposition}

\newtheorem{obs}{Observation}

\theoremstyle{remark}
\newtheorem{remark}{Remark}

\newcommand\numberthis{\addtocounter{equation}{1}\tag{\theequation}}
\DeclareMathOperator{\sgn}{sgn}
\DeclareMathOperator{\tr}{tr}

\DeclareMathOperator{\rk}{rk}
\DeclareMathOperator{\brk}{\underline{rk}}

\def\CC{\mathbb{C}}
\def\ZZ{\mathbb{Z}}
\def\ot{\otimes}

\title{Plethysm and fast matrix multiplication}
\author{Tim Seynnaeve}


\begin{document}

\begin{abstract}
	Motivated by the symmetric version of matrix multiplication we study the plethysm $S^k(\mathfrak{sl}_n)$ of the adjoint representation $\mathfrak{sl}_n$ of the Lie group $SL_n$. In particular, we describe the decomposition of this representation into irreducible components for $k=3$, and find highest-weight vectors for all irreducible components. Relations to fast matrix multiplication, in particular the Coppersmith-Winograd tensor, are presented.
\end{abstract}

\maketitle

\section{Introduction}
In 1969 \cite{strassen1969gaussian} Strassen presented his celebrated algorithm for matrix multiplication breaking for the first time the naive 
complexity bound of $n^3$ for $n\times n$ 
matrices. Since then, the complexity of the optimal matrix multiplication algorithm is one of the central problems in computer science. 
In terms of algebra we know that this question is equivalent to estimating rank or border rank of a specific tensor $M_{n,n,n}\in\CC^{n^2}\ot\CC^{n^2}\ot\CC^{n^2}$
\cite{JM1, landsberg_2017, BurgisserBook}. 
The current best lower and upper bounds are presented in \cite{JaJMSIAGA, JaJMIMRN, JMOttaviani, Virgi, LeGall}. 

We recall
that the constant $\omega$ is defined as the smallest number such that for any $\epsilon >0$ the multiplication of $n\times n$ matrices can be performed in time 
$O(n^{\omega+\epsilon})$.
Further, recall that the Waring rank of a homogeneous polynomial $P$ of degree $d$ is the smallest number $r$ of linear forms $l_1,\dots,l_r$ such that
$P=\sum_{i=1}^r l_i^d$. 
Recently, Chiantini et al. \cite{chiantini2017polynomials} provided another equivalent interpretation of $\omega$ in terms of Waring (border) rank. Namely, let $SM_n$ be a cubic in 
$S^3(\mathfrak{sl}_n^*)$ given by 
$SM_n(A)=\tr(A^3)$. Then $\omega$ is the smallest number such that for any $\epsilon>0$ the Waring rank (or Waring border rank) of 
$SM_n$ is $O(n^{\omega+\epsilon})$.
This observation was the initial motivation for our study of the plethysm $S^3(\mathfrak{sl}_n)$. 

The computations of plethysm are in general very hard and
explicit formulas are known only in specific cases \cite{macdonald1998symmetric}. For example for symmetric power $S^3(S^k)$ the decomposition was classically computed already in 
\cite{Thrall,Plunkett}, but $S^4(S^k)$ and $S^5(S^k)$ were only recently explicitely obtained in \cite{JaThomas}. As symmetric powers (together with exterior powers) are the simplest Schur functors, 
one could expect that respective 
formulas for $S^d(\mathfrak{sl}_n)$ are harder. In principle, one could use the methods of \cite{Howe, JaThomas, JaManivel} to decompose this plethysm, but this requires
a lot of nontrivial character manipulations. Instead, we present a very easy proof of explicit decomposition based on Cauchy formula and Littlewood-Richardson rule in 
Theorem \ref{thm:plet}. In fact, using our method one can inductively obtain the formula for $S^k(\mathfrak{sl}_n)$ for any $k$.

While matrix multiplication is represented by the (unique) invariant in $S^3(\mathfrak{sl}_n)$ the aim of this article is to understand the other highest-weight
vectors. A precise description of them is presented in Section \ref{sec:hwv}. We plan to undertake a detailed study of ranks and border ranks of other highest-weight vectors in
future work. Here we present just the first two nontrivial instances. It turns out, that two of the highest-weight vectors are (isomorphic to) the (four and five dimensional) variants of the
Coppersmith-Winograd tensor \cite{CW}. We recall that the best upper bounds for rank and border rank are based on a beautiful technique by Coppersmith and Winograd applied to 
a specific tensor $T$ \cite{Virgi}. While $T$ is extremely efficient for this technique, it is completely not clear which properties of $T$ make it so useful and
how to identify potentially better tensors. In fact, there are whole programs, see e.g.~\cite{cohn2003group}, aimed at finding tensors similar to, but better than Coppersmith-Winograd. We hope that 
other highest-weight vectors will also reveal their importance.    

\subsection*{Acknowledgement} The author would like to thank his advisor, Mateusz Micha\l{}ek, for the many helpful comments and discussions.

\section{The plethysm}
In this section we describe a general procedure to decompose $S^k(\mathfrak{gl}_n)$ and $S^k(\mathfrak{sl}_n)$ into irreducibles.
Recall that the irreducible representations of $SL_n$ are precisely the representations $\mathbb{S}_{\lambda}(\mathbb{C}^n)$, where $\lambda=[\lambda_1,\ldots,\lambda_{n-1}]$ is a partition of length at most $n-1$, and $\mathbb{S}_{\lambda}$ is the Schur functor associated to the partition $\lambda$ (consult for example \cite{FH13}).
\begin{theorem}\label{thm:plet}
For $n \in \mathbb{N}$, it holds that
\[
S^k(\mathfrak{gl}_n) \cong \bigoplus_{\lambda \vdash k}\bigoplus_{\nu}N_{\lambda \overline{\lambda}}^{\nu}\mathbb{S}_{\nu}(\mathbb{C}^n) \numberthis \label{plethysm}
\]
as $SL_n$-representations.
Here the second summation is over all partitions $\nu$ of length at most $n-1$, $N_{\lambda \mu}^{\nu}$ are the Littlewood-Richardson coefficients, and $\overline{\lambda}=[\lambda_1,\lambda_1-\lambda_{n-1},\ldots,\lambda_1-\lambda_2]$.
\end{theorem}
\begin{proof}
Note that $\mathfrak{gl}_n \cong (\mathbb{C}^n) \otimes (\mathbb{C}^n)^*$ as $SL_n$-representations. So
\begin{align*}
S^k(\mathfrak{gl}_n) \cong& S^k\big((\mathbb{C}^n) \otimes (\mathbb{C}^n)^*\big)
\cong \bigoplus_{\lambda \vdash k}\mathbb{S}_{\lambda}(\mathbb{C}^n) \otimes \mathbb{S}_{\lambda}(\mathbb{C}^n)^*\\
\cong& \bigoplus_{\lambda \vdash k}\mathbb{S}_{\lambda}(\mathbb{C}^n) \otimes \mathbb{S}_{\overline{\lambda}}(\mathbb{C}^n)
\cong \bigoplus_{\lambda \vdash k}\bigoplus_{\nu}N_{\lambda \overline{\lambda}}^{\nu}\mathbb{S}_{\nu}(\mathbb{C}^n) \text{.}
\end{align*}
The second isomorphism holds by Cauchy's formula; for the third one see for example \cite[15.50]{FH13}; the fourth isomorphism is the Littlewood-Richardson rule.
\end{proof}
To compute the decomposition of $S^k(\mathfrak{sl}_n)$, we simply note that 
\begin{align*}
S^k(\mathfrak{gl}_n) \cong& S^k(\mathfrak{sl}_n\oplus \mathbb{C}) \cong \mathbb{C} \oplus \bigoplus_{i=1}^k{S^i(\mathfrak{sl}_n)} \text{.}
\end{align*}
This allows us to compute the decomposition of $S^k(\mathfrak{sl}_n)$ inductively.\\
As a corollary we present an explicit decomposition in the case $k=3$. Computing the Littlewood-Richardson coefficients in \eqref{plethysm} gives us the decomposition of $S^3(\mathfrak{gl}_n)$ (resp.\  $S^3(\mathfrak{sl}_n)$) into irreducibles. We present these in Table \ref{tablePlethysm}: the first column lists the highest weights $\lambda$ of the occurring irreducible representations $\mathbb{S}_{\lambda}(\mathbb{C}^n)$. To be more precise: the first column actually shows the highest weights when we view $S^3(\mathfrak{gl}_n)$ (resp.\  $S^3(\mathfrak{sl}_n)$) as a $GL_n$-representation. (Recall that weights of $GL_n$ are $n$-tuples $[\lambda_1,\ldots,\lambda_n]\in \ZZ^n$ with $\lambda_1 \geq \ldots \geq \lambda_n$. The corresponding $SL_n$-weight is then $[\lambda_1-\lambda_n,\ldots,\lambda_{n-1}-\lambda_n]$.) 
The second and third column list the multiplicities of the irreducibles in $S^3(\mathfrak{gl}_n)$ resp.\  $S^3(\mathfrak{sl}_n)$. We also list the dimensions of the occurring irreducible representations $\mathbb{S}_{\lambda}(\mathbb{C}^n)$, as well as the dimensions of the projective homogeneous varieties contained in $\mathbb{P}(\mathbb{S}_{\lambda}(\mathbb{C}^n))$ (see Subsection \ref{subsec:homog}).
\begin{table}[h]
	\centering
	\caption{Irreducible components of $S^3(\mathfrak{gl}_n)$ and $S^3(\mathfrak{sl}_n)$}
	\label{tablePlethysm}
	\begin{tabular}{|l|l|l|l|l|}
		\hline
		Highest weight & $S^3(\mathfrak{gl}_n)$ & $S^3(\mathfrak{sl}_n)$ & Dimension & Variety \\ \hline
		$[0,\ldots,0]$ & $3$ & $1$ & $1$ & $0$ \\ \hline
		$[1,0,\ldots,0,-1]$ & $4$ & $2$ & $n^2-1$ & $2n-3$ \\ \hline
		$[2,0,\ldots,0,-2]$ & $2$ & $1$ & $\frac{(n-1)n^2(n+3)}{4}$ & $2n-3$ \\ \hline
		$[3,0,\ldots,0,-3]$ & $1$ & $1$ & $\frac{(n-1)n^2(n+1)^2(n+5)}{36}$ & $2n-3$ \\ \hline
		$[1,1,0,\ldots,0,-1,-1]$ & $2$ & $1$ & $\frac{(n-3)n^2(n+1)}{4}$ & $4n-12$ \\ \hline
		$[2,0,\ldots,0,-1,-1]$ & $1$ & $1$ & $\frac{(n-2)(n-1)(n+1)(n+2)}{4}$ & $3n-7$ \\ \hline
		$[1,1,0,\ldots,0,-2]$ & $1$ & $1$ & $\frac{(n-2)(n-1)(n+1)(n+2)}{4}$ & $3n-7$ \\ \hline
		$[2,1,0,\ldots,0,-1,-2]$ & $1$ & $1$ & $\frac{(n-3)(n-1)^2(n+1)^2(n+3)}{9}$ & $4n-10$ \\ \hline
		$[1,1,1,0,\ldots,0,-1,-1,-1]$ & $1$ & $1$ & $\frac{(n-5)(n-1)^2n^2(n+1)}{36}$ & $6n-27$ \\ \hline
	\end{tabular}
\end{table}
\subsection{Homogeneous varieties}\label{subsec:homog}
Let $V$ be an irreducible representation of a semisimple Lie group G. Then $\mathbb{P}V$ has a unique closed $G$-orbit $X$, which is the orbit of the highest-weight vector in $\mathbb{P}V$ under the action of $G$. The projective variety $X$ is isomorphic to $G/P$, where $P$ is a parabolic subgroup. We call these varieties homogeneous varieties or partial flag varieties. \\
In our case $G=SL_n$, we can compute the dimension of $X$ in the following way:
Consider the Dynkin diagram of $\mathfrak{sl}_n$, which consists of $n-1$ dots marked $1$ to $n-1$, and the Young diagram $\lambda$ associated to the representation $V$. For every $j \in \{1,\ldots,n-1 \}$, if the Young diagram has at least one column of length $j$, we remove the dot $j$ from the Dynkin diagram. After removing these dots the Dynkin diagram splits in connected components of size $k_i$. The dimension of our variety $X$ is then given by 
\[
\frac{1}{2}\left(n^2-n-\sum_{i}{(k_i^2+k_i)}\right) \text{.}
\]
This gives us the last column of Table \ref{tablePlethysm}.
\section{Highest weight vectors}\label{sec:hwv}
We now describe highest-weight vectors for all irreducible components of $S^3(\mathfrak{gl}_n)$. We write $E_{i,j} \in \mathfrak{gl}_n$ for the $n \times n$ matrix with as only nonzero entry a $1$ on position $(i,j)$. Note that the vector $E_{i,j}E_{i',j'}E_{i'',j''} \in S^3(\mathfrak{gl}_n)$ has weight $e_i+e_{i'}+e_{i''}-e_j-e_{j'}-e_{j''}$, where $e_i$ is the weight $[0,\ldots,1,\ldots,0]$ with a $1$ on the $i$-th position. Furthermore, to check that a weight vector $v$ in some representation $V$ of $SL_n$ is a highest-weight vector, it suffices to view $V$ as a representation of the Lie algebra $\mathfrak{sl}_n$ and check that every matrix $E_{i,i+1}$ acts by zero. Using this, it is straightforward to check that the vectors listed in Table \ref{tableHW} are indeed highest-weight vectors.
\begin{table}[h]
	\centering
	\caption{Highest weight vectors of $S^3(\mathfrak{gl}_n)$}
	\label{tableHW}
	\begin{tabular}{|l|p{60mm}|}
	\hline
	Weight & Highest Weight Vector \\ \hline
	$[0,\ldots,0]$ & $III$ \\ \hline
	$[0,\ldots,0]$ & $\sum_{i,j}{IE_{i,j}E_{j,i}}$ \\ \hline
	$[0,\ldots,0]$ & $\sum_{i,j,k}{E_{i,j}E_{j,k}E_{k,i}}$ \\ \hline
	$[1,0,\ldots,0,-1]$ & $IIE_{1,n}$ \\ \hline
	$[1,0,\ldots,0,-1]$ & $\sum_i{IE_{1,i}E_{i,n}}$ \\ \hline
	$[1,0,\ldots,0,-1]$ & $\sum_{i,j}{E_{1,n}E_{i,j}E_{j,i}}$ \\ \hline
	$[1,0,\ldots,0,-1]$ & $\sum_{i,j}{E_{1,i}E_{i,j}E_{j,n}}$ \\ \hline
	$[2,0,\ldots,0,-2]$ & $IE_{1,n}E_{1,n}$ \\ \hline
	$[2,0,\ldots,0,-2]$& $\sum_i{E_{1,n}E_{1,i}E_{i,n}}$ \\ \hline
	$[1,1,0,\ldots,0,-2]$ & $\sum_i{E_{1,n}E_{2,i}E_{i,n}-E_{2,n}E_{1,i}E_{i,n}}$ \\ \hline
	$[2,0,\ldots,0,-1,-1]$ & $\sum_i{E_{1,n}E_{1,i}E_{i,n-1}-E_{1,n-1}E_{1,i}E_{i,n}}$ \\ \hline
	$[1,1,0,\ldots,0,-1,-1]$ & $IE_{1,n}E_{2,n-1}-IE_{1,n-1}E_{2,n}$ \\ \hline
	$[1,1,0,\ldots,0,-1,-1]$ & $\sum_i{E_{1,n}E_{2,i}E_{i,n-1} - E_{2,n}E_{1,i}E_{i,n-1}}$ ${ - E_{1,n-1}E_{2,i}E_{i,n} + E_{2,n-1}E_{1,i}E_{i,n}}$ \\ \hline
	$[3,0,\ldots,0,-3]$ & $E_{1,n}E_{1,n}E_{1,n}$ \\ \hline
	$[2,1,0,\ldots,0,-1,-2]$ & $E_{1,n}E_{1,n-1}E_{2,n} - E_{1,n}E_{1,n}E_{2,n-1}$ \\ \hline
	$[1,1,1,0,\ldots,0,-1,-1,-1]$ & $\sum_{\sigma \in S_3}{\sgn{\sigma}E_{\sigma(1),n}E_{\sigma(2),n-1}E_{\sigma(3),n-2}}$ \\ \hline	
	\end{tabular}
\end{table}
\subsection{Waring rank and border Waring rank}
As explained in the introduction (see also \cite{chiantini2017polynomials}), estimating the (border) Waring rank of the highest-weight vector $\sum_{i,j,k}{E_{i,j}E_{j,k}E_{k,i}}$ is equivalent to determining the exponent $\omega$ of matrix multiplication. We will analyze the (border) Waring ranks of other highest-weight vectors. We start with the following surprising observation:
\begin{obs}
	Every highest-weight vector with weight different from $[0,\ldots,0]$ has Waring rank $O(n^2)$. Furthermore the weight space of $[0,\ldots,0]$ is 3-dimensional: it has a basis consisting of two vectors of Waring rank $O(n^2)$, and the vector $\sum_{i,j,k}{E_{i,j}E_{j,k}E_{k,i}}$.
\end{obs}
\begin{proof}
 Every of the highest-weight vectors in Table \ref{tableHW}, except for $\sum_{i,j,k}{E_{i,j}E_{j,k}E_{k,i}}$, is a sum of at most $n^2$ monomials, and every degree 3 monomial has Waring rank at most 4. 
\end{proof}
We now study the highest-weight vectors $IE_{1,n}E_{2,n-1}-IE_{1,n-1}E_{2,n}$ and\\ $E_{1,n}E_{1,n-1}E_{2,n} - E_{1,n}E_{1,n}E_{2,n-1}$, which we will rewrite as $xyz-xwt$ and $xzt-x^2y$.
\begin{prop}
	The cubics $f_1=xyz-xwt$ and $f_2=xzt-x^2y$ are two variants of the Coppersmith-Winograd tensor. Their ranks and border ranks (equal to Waring rank resp.\ Waring border rank) are given by 
	$\rk(f_1)=9,\brk(f_1)=6,\rk(f_2)=7,\brk(f_2)=4$.
\end{prop}
\begin{proof}
	After the change of basis $x=x_0$, $y=x_1+ix_2$, $z=x_1-ix_2$, $w=x_3+ix_4$, $t=-x_3+ix_4$, our cubic $f_1$ becomes $x_0x_1^2+x_0x_2^2+x_0x_3^2+x_0x_4^2$, which is precisely the Coppersmith-Winograd tensor $T_{4,CW}$ (here we use the notation from \cite[Section 7]{LaMM}).
	For $f_2$ we can do a similar change of basis, or alternatively we can use the geometric characterization of Coppersmith-Winograd tensors form \cite[Theorem 7.4]{LaMM}. We find that $f_2$ is isomorphic to $\tilde{T}_{2,CW}$.\\
	The ranks and border ranks of Coppersmith-Winograd tensors are known: consult for example \cite{CW} for the border ranks and \cite[Proposition 7.1]{LaMM} for the ranks.
\end{proof}
\begin{remark}
	The highest-weight vectors that are monomials are easily understood: $III$ and $E_{1,n}E_{1,n}E_{1,n}$ trivially have Waring rank equal to 1; $IIE_{1,n}$ and $IE_{1,n}E_{1,n}$ agree with the Coppersmith-Winograd tensor $T_{1,CW}$, hence have Waring rank 3 and border Waring rank 2.
\end{remark}
\bibliography{mybib}{}
\bibliographystyle{plain}
\end{document}